\documentclass{amsart}
\usepackage{amssymb}
\usepackage{ifthen}

\newtheorem{theorem}{Theorem}[section]
\newtheorem{lemma}[theorem]{Lemma}

\theoremstyle{definition}   
\newtheorem{remark}[theorem]{Remark}
\newtheorem{definition}[theorem]{Definition}
\newtheorem{claim}[theorem]{Claim}

\def\E{\mathbb{E}}

\def\beq{\begin{equation}}
\def\eeq{\end{equation}}

\newcommand{\Otimes}{\hspace{2pt}\bar\otimes\hspace{2pt}}
\newcommand{\tr}{\vartriangleright}
\newcommand{\tl}{\vartriangleleft}
\newcommand{\C}[1]{{\mathcal{#1}}}
\newcommand{\B}[1]{{\bf #1}}
\newcommand{\I}[1]{{\mathbb #1}}
\newcommand{\e}{\varepsilon}
\newcommand{\dd}{\,\mathrm{d}}
\newcommand{\citeB}[1]{\ifthenelse{\equal{#1}{}}{\cite{bogachev:mt}}{\cite[#1]{
bogachev:mt}}}
\newcommand{\citeBCL}[1]{\ifthenelse{\equal{#1}{}}{\cite{borgs+chayes+lovasz:10}
}{\cite[#1]{borgs+chayes+lovasz:10}}}
\newcommand{\citeJ}[1]{\ifthenelse{\equal{#1}{}}{\cite{janson:11}}{\cite[#1]{
janson:11}}}

\begin{document}

\title{Poset limits can be totally ordered}
\author{Jan Hladk\'y}
\address{Mathematics Institute, University of Warwick, 
Coventry, CV4 7AL, United Kingdom}
\email{J.Hladky@warwick.ac.uk}
\thanks{J.H.\ was supported by an
EPSRC fellowship.}

\author{Andr\'as M\'ath\'e}
\address{Mathematics Institute, University of Warwick, 
Coventry, CV4 7AL, United Kingdom}
\email{A.Mathe@warwick.ac.uk}
\thanks{A.M.\ 
 was supported by the EPSRC (grant EP/G050678/1) and the Hungarian Scientific Research Fund (grant 72655).}
\author{Viresh Patel}
\address{School of Mathematics,
Birmingham University, Edgbaston, 
Birmingham B15 2TT, United Kingdom}
\email{viresh.s.patel@googlemail.com}
\thanks{V.P.\ was supported by the EPSRC (grant EP/J008087/1).}

\author{Oleg Pikhurko}
\address{Mathematics Institute, University of Warwick, 
Coventry, CV4 7AL, United Kingdom}
\urladdr{http://homepages.warwick.ac.uk/staff/O.Pikhurko/}
\thanks{O.P.\ was supported by the European Research Council
(grant agreement no.~306493)
and the National Science Foundation of the USA (grant DMS-1100215).}

\begin{abstract}
S.~Janson [\emph{Poset limits and exchangeable random posets}, Combinatorica
\textbf{31} (2011), 529--563] defined limits of finite posets in
parallel to the emerging theory of limits of dense graphs. 

We prove that each poset limit can be represented as a kernel on the unit
interval with the standard order, thus answering an open question of Janson.
We provide two proofs: real-analytic and combinatorial. The
combinatorial proof is based on a Szemer\'edi-type Regularity Lemma for posets
which may be of independent interest.

Also, as a by-product of the analytic proof, we show that every atomless ordered probability
space admits a measure-preserving and almost order-preserving map to the unit interval.
\end{abstract}

\maketitle

\section{Introduction}

Given a class $\C C$ of finite structures and some measure $t(F,G)$ for
$F,G\in\C C$ of how
frequently $F$ appears in $G$ as a substructure, one can say that a
sequence
$\{G_n\}_{n\in\I N}$ \emph{converges} if $\{t(F,G_n)\}_{n\in\I N}$ converges
for every $F\in \C C$.

For example, if $\C C$ consists of finite graphs and $t$ denotes
the subgraph density, then we obtain the convergence of (dense) graphs
whose systematic study was initiated by Lov\'asz and
Szegedy~\cite{lovasz+szegedy:06} and Borgs et al~\cite{BCLSV:08}. In particular,
Lov\'asz and Szegedy~\cite{lovasz+szegedy:06} showed that for every
convergent sequence $\{G_n\}_{n\in\I N}$ of graphs there is a measurable
function $W:[0,1]^2\to[0,1]$ (called \emph{graphon}) 
such that for every graph
$F$ the limit of
$t(F,G_n)$ as $n\to\infty$ is a certain integral involving $W$. In fact, many
other
parameters of $G_n$ can be well approximated as $n\to\infty$
if we know $W$. This opens a general way of bringing analytic methods into the
study of large finite graphs. Connections to other areas are established
by alternative representations of ``graph limits'': by reflection positive graph
parameters
(Lov\'asz and Szegedy~\cite{lovasz+szegedy:06}), by positive flag algebra
homomorphisms (Razborov~\cite{razborov:07}), and
 by
partially exchangeable random arrays (Diaconis and
Janson~\cite{diaconis+janson:08}).

The theory of graph limits has received a great deal of attention and has been
extended to other structures as well, such as
hypergraphs (Elek and Szegedy~\cite{elek+szegedy:12}, see also
Tao~\cite{tao:07} and Austin~\cite{austin:08}),
permutations (Hoppen et al~\cite{HKMRS:13,HKMS:11:arxiv}), functions
on compact Abelian groups (Szegedy~\cite{szegedy:10:arxiv}), and others.

An analogous theory for limits of \emph{posets} (i.e.\ partially ordered sets)
was
initiated by Brightwell and Georgiou~\cite{brightwell+geogiou:10} and further
developed by Janson~\cite{janson:11}. Let us state some of these results.

We represent a poset as a
pair $(P,\prec)$ where $P$ is a finite \emph{ground set} and $\prec$ is a
\emph{strict
order relation} (i.e.\ it is transitive and no $a\in P$ satisfies $a\prec a$).

A map $f:P\to Q$ (not necessarily injective) is a \emph{homomorphism}
from $(P,\prec)$ to $(Q,\ll)$ if we have $f(x)\ll f(y)$ for every $x,y\in
P$ with $x\prec y$. The \emph{density} $t(\,(P,\prec),(Q,\ll)\,)$ is
the number of homomorphisms from $(P,\prec)$ to $(Q,\ll)$ divided by the
total number of possible maps $P\to Q$.
In other words, it is the probability that
a random map $P\to Q$ between the ground sets preserves the order relation.

\begin{definition}
\label{conv} 
A sequence of posets
$\{(P_n,\prec_n)\}_{n\in\I N}$ \emph{converges} if
$|P_n|\to\infty$ and
 \beq\label{eq:conv}
 \big\{\,t(\,(P,\prec),(P_n,\prec_n)\,)\,\big\}_{n\in\I N}\ \mbox{ converges
for every poset
$(P,\prec)$.}
 \eeq
\end{definition}

\begin{remark}\label{ConvVers} It is not hard to show (cf
\cite[Section~2.4]{lovasz+szegedy:06}) that Definition~\ref{conv}
does not change if we modify $t$ to be the density of induced and/or injective
homomorphisms.
\end{remark}

The potential usefulness of (\ref{eq:conv}) comes from the result
of 
Janson~\citeJ{Theorem 1.7} that for each convergent sequence there
is an analytic limit object as follows. (See Section~\ref{notation}
for an overview of the measure theory notation that we use.)

\begin{definition}\label{def:OrdProbSp}
An \emph{ordered probability space} $(S, \C F, \mu, \tl)$ is a probability
space $(S,\C F,\mu)$
equipped with a strict order relation $\tl$ such that $\{(x,y): x\tl y\}$
is an $\C F\otimes\C F$-measurable subset of $S\times S$.
\end{definition}

\begin{definition}\label{def:posetkernel}
A \emph{(poset) kernel} is a 5-tuple $(S, \C F, \mu, \tl,W)$, where
$(S, \C F, \mu, \tl)$ is
an ordered probability space and $W$ is
an $\C F\otimes\C F$-measurable function $S\times S\to[0, 1]$ such that, for
all $x,y,z\in S$,
\begin{eqnarray}
  W(x, y) > 0 &\Rightarrow & x \tl y,\label{axiom1}\\
  W(x, y) > 0\mbox{ \  and \ } W(y, z) > 0 &\Rightarrow& 
W(x, z) = 1.\label{axiom2}
\end{eqnarray}
\end{definition}

In particular, it follows from Definition~\ref{def:posetkernel} that
$W(x,y)W(y,x)=0$ for every $x,y\in S$. 

When no confusion arises,
we may abbreviate $(P,\prec)$ to $P$ and 
$(S, \C F, \mu, \tl,W)$ to $W$. Also, we will usually say
``kernel'' instead of ``poset kernel''.

\begin{definition} 
The \emph{density} of a poset $(P,{\prec})$ in a kernel $(S, \C F, \mu, \tl,W)$
is
 \beq\label{density}
 t(P,W):=\int_{S^{|P|}} \prod_{a,b \in P\atop a\prec b} W(x_a,x_b) \prod_{a\in
P}\dd\mu(x_a).
 \eeq 
\end{definition}

There is some analogy between $t(P,Q)$ and $t(P,W)$.
Namely, one
can interpret the expression in the right-hand side of (\ref{density}) as
follows. Select random elements $x_a$ 
of $(S,\C F,\mu)$ indexed by $P$ and
let 
$x_a\ll x_b$ with probability $W(x_a,x_b)$, with all choices being
mutually independent. Then $t(P,W)$ is exactly
the probability that $x_a\ll x_b$ for all $a\prec b$ in $P$. In fact, 
the connection is much deeper as the following result shows.

\begin{theorem}[Janson~\citeJ{Theorems~1.7 and 1.9(ii)}]\label{WExists} For
every convergent
sequence $\{P_n\}_{n\in\I N}$ of posets there is a kernel $(S,\C
F,\mu,\tl,U)$
such that 
 \beq\label{eq:WExists}
 t(P,U)=\lim_{n\to\infty} t(P,P_n),\quad\mbox{for every poset $P$.}
 \eeq
 Moreover, we can assume in (\ref{eq:WExists}) that
 \beq\label{1.9(ii)}
 (S,\C F,\mu)=([0,1],\C B,\lambda)
 \eeq
 is the unit interval
with the Lebesgue measure $\lambda$ on the Borel $\sigma$-algebra $\C B$.
\end{theorem}

Also, the converse of Theorem~\ref{WExists} was established in~\citeJ{}: for
every kernel $U$ there is a sequence of posets $\{P_n\}_{n\in\I N}$ that
satisfies~(\ref{eq:WExists}). In fact, the sampling procedure
informally described  after (\ref{density}) yields such a sequence with 
probability 1. 

Although we can require that (\ref{1.9(ii)}) holds, the proof 
in~\citeJ{} gives no
control over the order relation
$\tl$. This prompted
Janson~\citeJ{Problem 1.10} to ask if one can always take $([0,
1],\C B,\lambda,{<})$ with the standard order $<$ in
Theorem~\ref{WExists}. 
In a later paper~\cite{janson:12}, Janson answered his question for convergent
sequences of interval orders (see also~\citeJ{Theorem~1.9(iii)} for
a related result). Here we give the affirmative answer
in the general case. 

\begin{theorem}\label{our}
For every convergent
sequence $\{P_n\}_{n\in\I N}$ of posets there is a kernel $([0,1],\C
B,\allowbreak\lambda,<,U)$ 
such that (\ref{eq:WExists}) holds.
\end{theorem}

In fact, we provide two different proofs of Theorem~\ref{our}.

One goes via a Regularity Lemma for  posets that we prove
in Section~\ref{sec:finiteRL}. Our lemma finds a partition
$P=V_1\cup\dots\cup V_k$ which is $\e$-regular with respect to the
underlying graph of $(P,{\prec})$ and has the additional property that all
$\prec$-relations between parts go ``forward'' only. Having
such a Regularity Lemma, we follow the method of Lov\'asz and
Szegedy~\cite{lovasz+szegedy:06,lovasz+szegedy:07:gafa} to construct a kernel
$U$ by taking the ``limit'' of $\e$-regular partitions as $\e\to \infty$. The
above ``forward'' property allows us to ensure that 
$U(x,y)=0$ whenever $x\ge y$, thus proving Theorem~\ref{our}. We expect
our Regularity Lemma to have further applications.

The other proof of Theorem~\ref{our} is real-analytic. Actually, we
prove a somewhat stronger result (Theorem~\ref{th:main} below). In order
to state it, we have to give some further definitions.

Let
$(S,\C F,\mu,\tl,W)$ be a kernel. We call it \emph{strict} if $W(x,y)>0$ 
for every $x,y\in S$ with $x\tl y$. (Thus a kernel is strict if
the two sides of (\ref{axiom1}) are equivalent.)
Kernel axioms imply that if we define
$\tl'\,:=\{(x,y)\in S^2: W(x,y)>0\}$ then
$(S,\C F,\mu,\tl')$ is an ordered probability space on which $W$ is
a strict kernel. Clearly, this change does not affect~(\ref{density}). Thus,
we can addionally assume in Theorem~\ref{WExists} that $U$ is strict, see
\citeJ{Remark~1.2}.

\begin{definition}
An \emph{inclusion} between ordered probability spaces $(S,\C F,\mu,\tl)$
and $(S',\C F',\allowbreak\mu',\tl')$ is a measure-preserving function $f:S\to
S'$ such
that the set
 \beq\label{wrong}
 \left\{\,(x,y)\in S^2: x\tl y,\ f(x)\not\tl' f(y)\,\right\}
 \eeq
 has $\mu\otimes
\mu$-measure zero. Additionally, if we have a kernel $U$ on $(S',\C
F',\mu',\tl')$ then its \emph{pull-back along $f$} is the function $U^f:S^2\to[0,1]$,
defined by $U^f(x,y):=U(f(x),f(y))$ for $x,y\in S$. 
\end{definition}

\begin{theorem}\label{th:main} 
For every strict kernel $(S,\C F,\mu,\tl,W)$ such that $(S,\C F,\mu)$ is
atomless, there is a
kernel $([0,1],\C B,\lambda,<,U)$ and an inclusion
$f: (S,\C F,\mu,\tl)\to([0,1],\C B,\lambda,{<})$ such
that $W$ is equal to the pull-back $U^f$ almost everywhere.
\end{theorem}

Since $f$ in Theorem~\ref{th:main} is measure-preserving, we necessarily have
that $t(P,U)=t(P,W)$ for every poset $P$, that is, $U$ and $W$ represent the
same poset limit. Thus Theorem~\ref{our} follows
from Theorems~\ref{WExists} and~\ref{th:main}.

The notion
of a pull-back plays an important role in the theory of graphons.
Hopefully, our Theorem~\ref{th:main} will
be generally useful when studying poset kernels. For example, if the 
studied kernel property behaves well with respect to taking pull-backs, then one
can operate with the function $U$ that satisfies
Theorem~\ref{th:main} instead of the 5-tuple $(S,\C F,\mu,\tl,W)$. 
 
Given an ordered probability space $(S,\C F,\mu,\tl)$, the
indicator function $I_\tl$ of the order relation $\tl$ is clearly 
a strict kernel on it. Thus Theorem~\ref{th:main} has the following
direct corollary.

\begin{theorem}\label{Total} Every atomless ordered
probability space $(S,\C F,\mu,\tl)$
can be included into 
$([0,1],\C B,\lambda,<)$.\qed\end{theorem}
 
Theorem~\ref{Total} can be viewed as a measure theoretic analogue of the
statement that
every poset can be totally ordered. While extending this to infinite
partially ordered sets is an easy application of Zorn's lemma, the main content
of 
Theorem~\ref{Total} is that this total ordering can be done in a ``measurable''
way. Interestingly, the limit 
of totally ordered increasing posets happens to be our universal target space
$([0,1],\C B,\lambda,{<})$
with the indicator function $I_<$ as its kernel.\medskip

This paper is organised as follows. 
Section~\ref{notation} describes the measure theory notation that we 
frequently use. Section~\ref{aux} presents some auxiliary analytic
lemmas, thus making the flow of arguments in the later sections
smoother. Although Theorem~\ref{Total} is a direct consequence of
Theorem~\ref{th:main},
we prove it first in Section~\ref{pf:Total}. Then, in Section~\ref{pf:main}, we
show how Theorem~\ref{Total-sep} (a version
of  Theorem~\ref{Total}) implies
Theorem~\ref{th:main}. Our Regularity Lemma for posets is stated and proved in
(combinatorial) Section~\ref{sec:finiteRL} which can be read independently of
the other sections. We show how this Regularity Lemma gives an
alternative proof of
Theorem~\ref{our} in Section~\ref{ssec:altproof}.
Finally, Section~\ref{conclusion} contains some concluding remarks,
including examples that certain strengthenings of our results are 
not possible.

\section{Measure theory notation}\label{notation}

Let us give some notation that we are going to use frequently.
We do not define many standard concepts of measure theory but
refer
the reader to Bogachev's book~\citeB{} whose notation we
generally follow. We try to provide sufficient references
and explanations so that this paper is accessible to
combinatorialists who do not have a strong background in
measure theory.

Let $\I N=\{1,2,\dots\}$ and $\I R$ be the sets of respectively
natural and real numbers. When we consider a subset of $\I R$,
typically the unit interval $[0,1]$,
we will denote the $\sigma$-algebra of its Borel subsets by $\C B$ and
the Lebesgue measure by~$\lambda$. For a family $\C X$ of sets, let $\sigma(\C
X)$ denote the $\sigma$-algebra
generated by $\C X$. Let $I_X$ denote the indicator function of a set $X$ (that
is, $I_X(x)$ is
$1$ if $x\in X$ and 0 otherwise).

Let $(S,\C F,\mu)$ be a probability
space. As it is standard in measure theory, a real-valued function $f$ on $S$ is
called \emph{$\C F$-measurable} if it is $(\C F,\C B)$-measurable.
We denote by $\C F_\mu$ the completion of 
$\C F$ with respect to the measure $\mu$. 

We say that a property holds 
\emph{$(\C F,\mu)$-almost
everywhere} (and abbreviate this to \emph{$(\C F,\mu)$-a.e.}) if the set of
points of $S$ where it fails belongs to $\C F_\mu$ and has
$\mu$-measure zero.
When the underlying measure space is understood, we just
write ``a.e.''
In some rare cases when we consider more than one $\sigma$-algebra on 
the same set, the
bare term ``a.e.''\ refers to the largest $\sigma$-algebra.

We call two sets or two
functions \emph{equivalent} (and use the symbol $\sim$) if
they coincide a.e. The \emph{Fr\'echet-Nikodym distance}
between two sets $A,B\in\C F_\mu$ is $\mu(A\bigtriangleup B)$; it is
in general a pseudo-metric (it satisfies the Triangle Inequality
but may evaluate to $0$ for $A\not=B$).
The space $(S,\C F,\mu)$ is called \emph{separable} if $\C F$ has a countable
subset which is
dense with respect
to the Fr\'echet-Nikodym distance.

Let $\C A\subseteq\C F$ be another $\sigma$-algebra and let $f$
be an integrable real-valued function on $(S,\C F,\mu)$. The \emph{conditional 
expectation} $\E (f|\C A)$ is the set of all $\C A$-measurable functions
$g:S\to\I R$ such that for every bounded $\C A$-measurable function $h:S\to \I
R$
we have
 \beq\label{CondExp}
 \int h(x)g(x)\dd\mu(x) = \int h(x)f(x)\dd\mu(x).
 \eeq
 As it is well-known, $\E (f|\C A)\not=\emptyset$ and every two functions in $\E
(f|\C A)$ are
equivalent; also, it is enough to
check~(\ref{CondExp}) for $\{0,1\}$-valued $h$ only, i.e.\ that
  \beq\label{CondExpSet}
 \int_A g(x)\dd\mu(x) = \int_Af(x)\dd\mu(x)\qquad\mbox{for all $A\in\C A$}.
 \eeq
We refer the reader to \citeB{Section~10.1.1} for
some basic properties of
conditional expectation. We may treat $\E(f|\C A)$ as a
single function (rather as a set of functions), when the studied property does
not depend
on the choice of a representative.

Let $(S',\C F',\mu')$ be another probability space.  A map $f:S\to S'$ is
\emph{measure-preserving} if $f$ is $(\C F,\C F')$-measurable and,
for every $A\in\C F'$, we have $\mu(f^{-1}(A))=\mu'(A)$.
The products of $\sigma$-algebras
and measures are denoted by $\C F\otimes\C F'$ and $\mu\otimes\mu'$.
We use the shorthand $\C F\Otimes\C F'$
for $(\C F \otimes\C F')_{\mu\otimes\mu'}$, the completion of $\C F\otimes\C F'$
with respect to $\mu\otimes\mu'$.
We will be using Fubini's theorem (\citeB{Theorem~3.4.4}) very
frequently, often without explicitly mentioning it.
Let us stress that one has to be careful when dealing with products of
$\sigma$-algebras and
measures. For example, the product of two complete measure spaces is not
complete in general.
 Also, see Exercises 44--45, 49--51, and 55 in \citeB{Section~3.10} for
counterexamples to
some 
``plausible'' statements related to Fubini's theorem.  

Let us give some kernel-specific definitions (when the underlying
ordered probability space $(S,\C F,\mu,\tl)$ is understood). For $A\subseteq
S$,
let $A^c:=S\setminus A$ denote the complement of $A$.
For $X\in\C F\otimes\C F$, we
define
 \beq\label{muprec}
 \mu_\tl(X):=\int_{X} I_\tl(a,b)\dd\mu(a)\dd\mu(b).
 \eeq
 For a 2-variable function $W:S^2\to\I R$ and $y\in S$, 
the \emph{slice function} $W_y:S\to \I R$ is defined by $W_y(x):=W(x,y)$. 
 We call $W:S\times
S\to[0,1]$ an \emph{almost (poset) kernel} 
if $W$ is $\C
F\Otimes\C F$-measurable and the kernel axioms (\ref{axiom1}) and
(\ref{axiom2})
hold for a.e.\ triple $(x,y,z)$ in $(S,\C F, \mu)^3$.

Although our Theorem~\ref{th:main} takes a kernel $W$ as input and then
produces another kernel $U$, we have to deal with almost kernels at
intermediate stages of the proof.
(For example, the pull-back $U^f$ in Theorem~\ref{th:main} is generally an
almost kernel.)

\section{Auxiliary analytic lemmas}\label{aux}

Here we present some auxiliary results that we will need later.

Janson~\citeJ{} proved that one can transform an almost
kernel $(S,\C B,\lambda,{\tl}, W)$ with $S\subseteq \I R$ into a
kernel $(S,\C B,\lambda,{\tl'},W')$ with $W'\sim W$ and some $\tl'$. We show
that
in the special
case of the unit interval with the standard order, one can also keep
the order relation
intact.

\begin{lemma}\label{clean01} Let $([0,1],\C B,\lambda,{<},U)$ be an almost
kernel. Then there is $U'\sim U$ such that $([0,1],\C B,\lambda,{<},U')$ is
a kernel.\end{lemma}
\begin{proof}
First, we choose a $\C B\otimes\C B$-measurable function $U_0\sim U$; it
exists by
\citeB{Proposition~2.1.11}. Then we proceed in a similar fashion
as is done by Janson \citeJ{Pages~547--548}, so we will be rather brief.
We refer the reader to \citeB{Section 5} for the definitions and basic
properties 
of Lebesgue and density points. Here, these are defined relative to the
domain of a function; namely, the system of shrinking neighbourhoods around
$(x,y)\in [0,1]^2$ is taken to be $(x\pm\e)\times (y\pm\e)$ as $\e\to0$, where
e.g.\ $(x\pm\e)$ denotes the intersection of the 
open interval $(x-\e,x+\e)$ with $[0,1]$.

We define $U_1 \colon [0,1]^2\to [0,1]$ by $U_1(x,y):=U_0(x,y)$ if
$(x,y)\in[0,1]^2$ is a Lebesgue point of $U_0$. Next, if $(x,y)$ is a density
point of
the set $\{(x,y) : U_0(x,y)=1\}$, then let $U_1(x,y):=1$. 
(Recall that a density point need not belong to the set itself.) For all other
pairs $(x,y)\in [0,1]^2$, we define
$U_1(x,y):=0$. Note that $U_1\sim U_0$ and $U_1$ is still $\C B\otimes\C B$-measurable.

We claim that $U_1$ is a kernel on $([0,1],\C B, \lambda, <)$. Suppose
that $U_1(x,y)>0$. Then for every sufficiently small $\e>0$, we
have $U_0(x',y')>0$
for most points $(x',y')\in (x\pm\e)\times (y\pm \e)$. In
particular,
$x'<y'$ for most of these pairs and therefore $x<y$. 

Now suppose that $U_1(x,y)>0$ and $U_1(y,z)>0$. 
Then for every sufficiently small $\e>0$, we have $U_0(x',y')>0$ for most
points $(x',y')\in (x\pm\e)
\times (y\pm\e)$ and $U_0(y',z')>0$ for most points
$(y',z')\in (y\pm\e) \times (z\pm\e)$.
This implies that  we have
$U_0(x',z')=1$ for most points $(x',z')\in (x\pm\e) \times (z\pm\e)$. Thus
$(x,z)$ is a density point of
$\{(x,y): U_0(x,y)=1\}$ and therefore $U_1(x,z)=1$, as required.
\end{proof}

\begin{lemma}[Borgs, Chayes, and Lov\'asz~\citeBCL{Lemma~3.4}]\label{bcl3.4}
Let $(S,\C F)$ and $(S',\C F')$ be measurable spaces, and
let $W:S\times S'\to \I R$ be a bounded $\C F\otimes\C F'$-measurable
function. Then there exist countably generated $\sigma$-algebras $\C
F_0\subseteq \C F$
and $\C F_0'\subseteq \C F'$ such that $W$ is $\C F_0\otimes\C
F_0'$-measurable.\qed\end{lemma}

\begin{lemma}\label{lm:push} Let $f:S\to S'$ be an inclusion 
of ordered probability spaces \mbox{$(S,\C F,\mu,\tl)$} and
$(S',\C
F',\mu',\tl')$. Let $W$ be a kernel
on $S$ such that $W\sim \E(W|\C A\otimes\C A)$, where $\C A:=f^{-1}(\C
F')$. Then there is an almost kernel $U$ on
$(S',\C
F',\mu',\tl')$ with $W\sim U^f$. 
 \end{lemma}

\begin{proof} We construct $U$ following the argument
of Borgs, Chayes, and Lov\'asz~\citeBCL{Lemma~3.1}.

Note that $f\times f:(S,\C F,\mu)^2\to(S',\C
F',\mu')^2$ is measure-preserving.
Define a measure $\nu$ on $\C F'\otimes\C F'$ by
 $$
 \nu(X):=\int_{(f\times f)^{-1}(X)} W(x,y)\dd\mu(x)
\dd\mu(y),\qquad X\in\C F'\otimes\C F'.
 $$
 This measure
$\nu$ is absolutely continuous with respect to $\mu'\otimes\mu'$. Hence, the
Radon-Nikodym derivative 
 $$
 U:=\frac{\dd\nu}{\dd(\mu'\otimes\mu')}
 $$
 exists (\citeB{Theorem~3.2.2}). Namely, $U:S'\times S'\to\I R$ is a
$\mu'\otimes\mu'$-integrable
function such that for every $X\in\C F'\otimes\C F'$ we have
$\nu(X)=\int_X U\dd(\mu'\otimes\mu')$. 

The last identity implies (given that $f$ is measure-preserving and that
$0\le W\le
1$) that the set 
 $$\{(x,y)\in
S'\times S':U(x,y)>1\mbox{ or } U(x,y)<0\}$$ 
 has measure zero. By changing $U$ on a null set, we can assume
that $U$ is $\C F'\otimes\C F'$-measurable (see \citeB{Proposition~2.1.11}) and
that
the values of $U$
belong to $[0,1]$.
 In particular, the
pull-back
$U^f$ is $\C A\otimes\C A$-measurable. Moreover,
for any $Y\in \C A\otimes \C A$, say $Y=(f\times f)^{-1}(X)$, 
we have that 
 \beq\label{integrals}
 \int_{Y} U^f \dd(\mu\otimes\mu)= \int_X U\dd(\mu'\otimes\mu') =\int_Y
W\dd(\mu\otimes\mu).
 \eeq
 By~(\ref{CondExpSet}), we conclude that $U^f\in\E(W|\C A\otimes\C A)$. Thus
$U^f$ is
a.e.\ equal to $W$ by the assumption of the lemma.

Let us verify that $U$ is an almost kernel. First, consider the set
 $$
 X:=\{(x,y)\in S'\times S': x\not\tl'y,\ U(x,y)>0\}\in \C F'\otimes\C
F'
 $$
 of points where the first kernel axiom (\ref{axiom1}) fails for $U$. By
(\ref{integrals}),
the integral of $U$ over $X$ is the
same as the integral of $W$ over $Y:=(f\times f)^{-1}(X)$. 
Since $f$ is an inclusion, we have $\mu_\tl(Y)=0$, where $\mu_\tl$
is defined by (\ref{muprec}). Since $W$ is
a
kernel, it is zero a.e.\ on $Y$.
It follows that $X$ has measure zero, that is, $U$ satisfies~(\ref{axiom1}) a.e.

Define $u(x,y,z):=U(x,y)\,U(y,z)\,(1-U(x,z))$. Since $U^f\sim W$
and $W$ is a kernel,  we have
$u^f\sim 0$. Since $f$ is measure-preserving, 
we have $\int u=\int u^f=0$. The non-negativity of $u$ implies
that $u\sim 0$, that is, $U$ satisfies (\ref{axiom2})
a.e.\end{proof}

\begin{remark} The conditional expectation of a kernel
need not be an almost kernel. For example, let $S:=\{a,b,c,d\}$ 
with $\C F:=2^S$ and $\mu$ being the uniform measure. Let $\C A=\sigma(\{a\},\{b,c\},\{d\})\subseteq
\C F$ be obtained by ``gluing'' $b$ and $c$ together. Let $a\tl
b$ and
$c\tl d$ be all order relations and let $W:=I_{\tl}$.
Then any $U\in \E(W|\C A\otimes\C A)$ satisfies
$U(a,b)=1/2$, $U(b,d)=1/2$, and
$U(a,d)=0$ and cannot be an almost kernel. Also, pull-backs
do not preserve (almost) kernels in general: for example, 
the pull-back of $I_<$ with respect to 
the identity inclusion of $([0,1],\C B,\lambda,\emptyset)$
into $([0,1],\C B,\lambda,<)$ does not satisfy (\ref{axiom1}).\end{remark}

\begin{lemma}\label{lm:sections} Let $(S,\C F,\mu)$ be a probability
space. Let $\C A\subseteq\C F$ be another $\sigma$-algebra such
that $(S,\C A,\mu)$ is separable. Let
$W:S^2\to\I R$ be a bounded $\C F\otimes\C F$-measurable function.
 Let $g\in\E(W | \C A \otimes \C F)$.  Then, for a.e.\ $y\in S$, we
have that $g_y\sim\E(W_y | \C
A)$.
\end{lemma}

\begin{proof}
 By definition, $g$ is $\C A\otimes
\C F$-measurable. It follows by~\citeB{Proposition~3.3.2} that
the slice function $g_y$ is $\C A$-measurable
for every $y\in S$.

Fix $A\in\C A$. By the definition of conditional expectation, we have
that $\int_{A\times B}g=\int_{A\times B}W$ for every
$B\in\C F$. Likewise, 
 \begin{equation}\label{Wy}
 \int_A \E(W_y | \C A)= \int_A W_y,\quad\mbox{for every $y\in S$.}
 \end{equation}
 By Fubini's
theorem, the
latter
function is integrable as a function of $y$. Moreover,
 $$
 \int_B \left(\int_A W_y(x) \dd\mu(x)\right)\dd\mu(y)= \int_{A\times B} W =
\int_{A\times
B}g= \int_B \left(\int_A g_y(x) \dd\mu(x)\right)\dd\mu(y).
$$
 Since $B\in\C F$ was arbitrary, \citeB{Corollary 2.5.4} gives that $\int_A 
W_y=\int_A g_y$ for a.e.\ $y$.
Let us remove all exceptional points $y$ when $A$ runs over a
dense countable subset $\{A_1,A_2,\dots\}\subseteq\C A$
in $(S,\C A,\mu)$ as well as those $y$ for which
$\|g_y\|_\infty>\|W\|_\infty$ or $\|W_y\|_\infty>\|W\|_\infty$.
It is easy to see that the remaining set $Y$ has measure 1.

Fix any $y\in Y$. For \emph{every} $A\in\C A$ we have that
 $$
 \left|\int_A W_y - \int_{A} g_y\right|\le \left|\int_A W_y - \int_{A_i}
W_y\right|+\left|\int_A g_y - \int_{A_i} g_y\right|\le 4\,
\|W\|_\infty\,\mu(A\bigtriangleup
A_i).
 $$
 Since the right-hand side can be made arbitrarily small by choosing a
suitable $A_i$, we
conclude that
$\int_A W_y = \int_A g_y$. 
Since $A\in\C A$ was arbitrary and
both $\E(W_y|\C A)$ and $g_y$
are $\C A$-measurable, they coincide a.e.\ by (\ref{Wy}).
The lemma is proved as $\mu(S\setminus Y)=0$.
\end{proof}

\section{Proof of Theorem~\ref{Total}}\label{pf:Total}

Let $(S,\C F,\mu,\tl)$ be given. Lemma~\ref{bcl3.4}, when applied to the
indicator function $I_\tl$,
returns
two countably generated $\sigma$-algebras $\C F_0,\C F_0'\subseteq \C F$.
Let $\C F':=\sigma(\C F_0\cup \C F_0')$ be the $\sigma$-algebra on $S$ generated
by $\C
F_0\cup\C F_0'$. 
 By enlarging a set of generators of $\C F'$ by adding a countably many
elements of $\C F$, we 
can additionally make $(S,\C F',\mu)$ atomless. 

Clearly,
if we prove Theorem~\ref{Total} for this new space $(S,\C F',\mu,\tl)$, then the
same inclusion $f$
will work
for the original one (as $\C F'\subseteq \C F$).
Thus, without loss of generality, let us assume that $\C F$ is countably
generated. It easily follows
(see e.g.\ Exercise~1.12.102 and its hint in~\citeB{}) that $(S,\C F,\mu)$ is
separable.
Thus it is enough to prove the following theorem (whose last claim
will be needed later in Section~\ref{pf:main}).

\begin{theorem}
\label{Total-sep} Let $(S,\C F,\mu,\tl)$ be an ordered probability
space
such that $(S,\C F,\mu)$ is atomless and separable. Then there is 
an inclusion  $f:(S,\C F,\mu,\tl)\to ([0,1],\C B,\lambda,<)$ such that
every set $A\in\C F$ with $\mu_\tl(A\times A^c)=0$ belongs to
$(f^{-1}(\C B))_{\mu}$, 
the completion of $f^{-1}(\C B)$ with respect to the measure $\mu$. 
\end{theorem}
 
So we prove Theorem~\ref{Total-sep} now.

\begin{claim}\label{cl:cutintotwo}
Let $B\in\C F$ with $\mu(B)>0$. Then there exists $A\in\C F$
such
that $\mu_\tl(A\times A^c)=0$, $\mu(B\cap A)>0$, and $\mu(B\cap A^c)>0$.
\end{claim}
\begin{proof}[Proof of Claim.]
 Let ${\tr_x}:=\{y\in S: y\tr x\}\in\C F$ be the \emph{strict upper shadow}
of $x\in S$ and let
 $$
 B':=\{x\in B:\mu(B\cap {\tr_x})>0\}.
 $$

First, suppose that $\mu(B')>0$. Clearly, $\mu_\tl(B'\times B')\le \mu(B')^2/2$.
By
Fubini's theorem, there is $x\in B'$ with $\mu(B'\cap{\tr_x})\le
\mu(B')/2$.                                        
Clearly, $A:={\tr_x}$ has the required properties. 

If $\mu(B')=0$, then $\mu_\tl(B\times B)=0$ by Fubini's theorem. Since $\C F$ is
atomless, it contains $A'\subseteq B$ with $0<\mu(A'\cap B)<\mu(B)$. 
The function $a(x):=\mu\big(\{y\in A': y\tl x\}\big)$ is $\C F$-measurable by
\citeB{Corollary~3.3.3}. 
The set $X:=\{x\in S: a(x)>0\}\in\C F$ is
clearly up-closed with respect to $\tl$ and it intersects $B$
in a set of measure $0$ by Fubini's theorem. It is easy to see
that $A:=A'\cup X$
satisfies
the claim.
\end{proof}

Let $$\C T:=\{A\in \C F: \,0<\mu(A)<1, \,\mu_\tl(A\times A^c)=0\}.$$ 
By Claim~\ref{cl:cutintotwo}, $\C T$ is non-empty and, moreover, infinite.
Since $(S,\C F,\mu)$ is separable, we can choose a countable subset
$\{A_1, A_2, \ldots\}\subseteq \C T$ which is dense in $\C T$
with respect to the Fr\'echet-Nikodym distance.

We define $f$ so that it satisfies the following properties:
\begin{equation}\label{informal}
 \begin{array}{rcl}
 f(A_1^c) & \subseteq&[0,\,\mu(A_1^c)],\\
 f(A_1) &\subseteq& [\mu(A_1^c),\,1],\\
f(A_1^c\cap A_2^c)&\subseteq& [0,\,\mu(A_1^c\cap A_2^c)],\\
f(A_1^c\cap A_2)&\subseteq& [\mu(A_1^c\cap A_2^c), \,\mu(A_1^c)],\\
f(A_1\cap A_2^c)&\subseteq& [\mu(A_1^c), \,\mu(A_1^c)+\mu(A_1\cap A_2^c)],\\
 f(A_1\cap A_2)&\subseteq& [\mu(A_1^c)+\mu(A_1\cap A_2^c), \,1],
 \end{array}
\end{equation}
and so on. Specifically, for a (finite or infinite) binary
sequence $\B b=(b_1,b_2,\dots)$,
let 
 \begin{eqnarray*}
A_{\B b}&:=&\bigcap \{A_i^c : b_i=0\} \cap \bigcap \{A_i: b_i=1\},\\
 S_{\B b} &:=& A_{\B b}\cup \bigcup \{A_{b_1,\dots,b_{i-1},0} : b_i=1\} \ = \
\bigcup \{A_{\B b'} : {\B b'\le_\text{lex} \B b}\},
 \end{eqnarray*}
where $\le_\text{lex}$
denotes the lexicographical order (which we apply only to two
binary sequences of the same length).
 Next, for $x\in S$ define $\B
b(x):=(b_1(x), b_2(x), \ldots)\in
\{0,1\}^{\I N}$ by $b_i=I_{A_i}$ for $i\in \I N$.
Thus
$\B b(x)$ is the unique infinite sequence with $x\in A_{\B b(x)}$.
Finally, we define 
 $$f(x):=\mu(S_{\B b(x)}).$$

\begin{claim}\label{measurable} The function $f$ is $\C
F$-measurable.\end{claim}

\begin{proof}[Proof of Claim.]
The function $(x,y)\mapsto I_{\le_\text{lex}}(\B b(x),\B b(y))$
is $\C F\otimes\C F$-measurable:
the pre-image of
$0$ is                                                       
 $$
 \cup_{i\in \I N}\cup_{b_1,\dots,b_{i-1}}
\left(A_{b_1,\dots,b_{i-1},1}\times A_{b_1,\dots,b_{i-1},0}\right)\in \C
F\otimes\C F.
 $$
 Thus $f(x)=\int I_{\le_\text{lex}}(\B b(y),\B b(x))\dd\mu(y)$ is
$\C F$-measurable by \citeB{Corollary~3.3.3}.\end{proof}

\begin{claim}\label{Ab} For every $a\in [0,1]$ and every infinite $\B b$, both
sets  $f^{-1}(a)$  and $A_{\B b}$ belong to $\C F$ and
have $\mu$-measure zero.
\end{claim}
 \begin{proof}[Proof of Claim.] We have
$f^{-1}(a)\in\C F$ by Claim~\ref{measurable} and $A_{\B b}\in\C F$ because each
$A_i$ is in $\C F$. Each of $f^{-1}(a)$  and $A_{\B b}$ is a null set because
otherwise, by Claim~\ref{cl:cutintotwo}, some $A_i$ would cut it into two parts of positive measure, which
is clearly impossible.\end{proof}

\begin{claim}\label{limsup}
Let $|\B b|$ denote the length of the sequence $\B b$. Then
$$\lim_{n\to\infty} \sup_{|\B b|=n} \mu(A_{\B b}) = 0.$$
\end{claim}
\begin{proof}[Proof of Claim.]
Assume to the contrary that this $\limsup$ is $\varepsilon>0$. Then, by K\"onig's
lemma (\cite[Lemma~8.2.1]{diestel:gt}), there exists an infinite sequence $\B b=(b_1, b_2, \ldots)$ such that
$\mu(A_{b_1, \ldots, b_n})\ge \varepsilon$ for every $n$. (Note that
$A_{b_1, \ldots, b_n} \supseteq A_{b_1, \ldots, b_{n+1}}$.) As $A_{\B
b}=\cap_{n=1}^\infty A_{b_1, \ldots, b_n}$, we conclude that
$\mu(A_{\B b})\ge
\varepsilon>0$, contradicting Claim~\ref{Ab}.
\end{proof}

\begin{claim}\label{bitdense}
The set $\{\mu(S_{\B b}) : |\B b| \text{ is finite}\}$ is dense in $[0,1]$.
\end{claim}
\begin{proof}[Proof of Claim.]
Consider the binary sequences of length $n$. Notice that, for any finite $\B b$,
$$\mu(S_{\B b}) = \sum_{\B b' \le_\text{lex} \B b} \mu(A_{\B b'}),$$
$\mu(S_{1,1, \ldots, 1})=1$, and that 
$\mu(S_{0,0, \ldots, 0})=\mu(A_{0,0,\ldots 0})$ (which tends to 0 by
Claim~\ref{Ab}).
Let $\B b'\le_\text{lex}\B b''$ be two sequences of length $n$ which are
consecutive in $\le_\text{lex}$.
Then $$\mu(S_{\B b''}) - \mu(S_{\B b'})= \mu(A_{\B b''}) \le \sup_{|\B b|=n}
\mu(A_{\B b}).$$
Combining this with Claim~\ref{limsup} gives the statement.
\end{proof}

\begin{claim}\label{mp2} The function $f$ is measure-preserving.\end{claim}

\begin{proof}[Proof of Claim.] Claim~\ref{bitdense} implies that the
intervals $[0, \mu(S_{\B b})]$, where
$\B b$ runs over finite binary sequences, generate the Borel $\sigma$-algebra.
Thus is enough to show that for every finite $\B b$ we have
$\mu\left(f^{-1}(\,[0,a]\,)\right) = a$, where $a:= \mu(S_{\B b})$.
 The latter identity follows from the fact that the symmetric difference of
$S_{\B b}$ and $f^{-1}\left(\,[0,a]\,\right)$ is a subset of
$f^{-1}(a)$ and therefore has measure zero by
Claim~\ref{Ab}.\end{proof}

The set $Y:=\{(x,y)\in S^2: x\tl y,\ f(x)>f(y)\}$ is a subset
of $\cup_{i=1}^\infty (A_i\times A_i^c)\in\C F\otimes\C F$. But the latter
set has $\mu\otimes \mu$-measure zero by the definition of $A_i$'s. Thus
$Y$ also has measure zero. Next, consider the set $Y_0:=\{(x,y)\in S^2:  f(x)=
f(y)\}$.
Since $f$ is $\C F$-measurable, we have $Y_0 \in \C F\otimes\C F$.
Every slice of $Y_0$ has measure zero
by
Claim~\ref{Ab}. By Fubini's theorem, $Y_0$ has itself measure
zero. 
We conclude that $f$ is an inclusion.

Finally, take an arbitrary $A\in\C F$ with $\mu_\tl(A\times A^c)=0$. 
For every $i\in\I N$ there is a set $A_{n_i}\in\C T$ such that
$\mu(A_{n_i}\triangle A)<2^{-i}$. Since
$A_{n_i}\bigtriangleup f^{-1}(X)\subseteq f^{-1}(Y)$, where
$X$ is some finite union of intervals and $Y$ is the set of
their endpoints, we have by Claim~\ref{Ab} that
$A_{n_i}$ is $(f^{-1}(\C B))_{\mu}$-measurable. This implies that
$A$ is $(f^{-1}(\C B))_{\mu}$-measurable: indeed, for 
 \begin{equation}\label{approx}
 A':=\limsup_{i\to\infty}
A_{n_i} = \bigcap_{k=1}^\infty \bigcup_{j=k}^{\infty} A_{n_j}\in (f^{-1}(\C
B))_{\mu}
 \end{equation}
 we have $\mu(A'\triangle A)=0$.
 This finishes the proof of
Theorem~\ref{Total-sep} (and Theorem~\ref{Total}).

\section{Proof of Theorem~\ref{th:main}}\label{pf:main}

As in Theorem~\ref{Total}, we
can assume that $\C F$ is separable. Apply Theorem~\ref{Total-sep} to $(S,\C
F,\allowbreak\mu,\tl)$ to obtain
an inclusion $f:S\to[0,1]$. As we will see later, the same $f$ will work in
Theorem~\ref{th:main}.
(Thus, rather interestingly, $f$ can be chosen independently of $W$
in Theorem~\ref{th:main} if $\C F$ is separable.)
Let
 $$
 \C A:=(f^{-1}(\C B))_\mu.
 $$
 Since
$f$ is $\C F$-measurable, we have that $\C A\subseteq \C F_\mu$.

We would like to apply Lemma~\ref{lm:push}. In order to do so, we have to verify
first that $W\sim E(W|\C A\otimes\C A)$.
(Note that $E(W | \C A \otimes \C A) \sim E(W | f^{-1}(\C B) \otimes f^{-1}(\C
B))$.)

\begin{claim}\label{cl:Wy} For every $y$, the slice function $W_y$ is $\C
A$-measurable.
\end{claim}
\begin{proof}[Proof of Claim.]
Fix any $a\in[0,1]$. For every $y\in S$, the set
 $$
 A:=W_y^{-1}(\,(a,1]\, )=\{x\in S: W(x,y)>a\}
 $$ 
 is in $\C F$. Since $W$
is a strict kernel, we have $\mu_\tl(A^c\times A)=0$ for every $y$. By the
second part
of
Theorem~\ref{Total-sep}, $A\in\C F$ belongs in fact to $\C A$. Since
intervals $(a,1]$ generate the Borel $\sigma$-algebra,
the claim follows.\end{proof}

The functions $W$ and $\E(W|\C A\otimes\C F_\mu)$ are both $\C F\Otimes\C
F$-measurable. (Note that $\C F\Otimes\C
F=\C F_\mu\Otimes\C
F_\mu$.)
Also,  their $y$-slices are
equivalent for a.e.\ $y$ by Lemma~\ref{lm:sections} and Claim~\ref{cl:Wy}. By Fubini's
theorem,
the subset of $S^2$ where these two functions differ has
$\mu\otimes\mu$-measure zero. In other words,  $W$ is $\C A\Otimes\C
F$-measurable and, by symmetry, $\C F\Otimes\C
A$-measurable.

\begin{claim}\label{cl:comm}
$$W\sim \E(W | \C A \otimes \C A).$$
\end{claim}

\def\Wt{\widetilde{W}}

\begin{proof}[Proof of Claim.]
We follow the argument of Borgs, Chayes, and Lov\'asz \citeBCL{Section~3.3.5}.
Let $\Wt\in\E(W|\C A\otimes \C A)$. It is enough to prove that for
every $A, B\in\C F$, 
 $$\int_{A\times B} W = \int_{A\times B} \Wt.$$
Take any $g_A\in \E(I_A | \C A)$ and $g_B\in \E(I_B | \C A)$. Define
 \begin{eqnarray*}
 U_A(y)&:=&\int_A W(x,y) \dd\mu(x)\ =\ \int W(x,y) I_A(x) \dd\mu(x),\\
 V_B(x)&:=&\int  W(x,y) g_B(y) \dd\mu(y).
 \end{eqnarray*}
Clearly, $g_A$ is $\C A$-measurable.
Since $W$ is $\C F\Otimes \C A$-measurable (as it was noted
after Claim~\ref{cl:Wy}), $U_A$ is
$\C A$-measurable by Fubini's theorem. Similarly, $V_B$ is also $\C
A$-measurable.
Repeatedly using Fubini's theorem and (\ref{CondExp}), 
we get
\begin{align*}
\int_{A\times B} W(x,y) \dd\mu(x) \dd\mu(y) &=
\int U_A(y) I_B(y) \dd\mu(y)\ =\\
\int U_A(y) g_B(y) \dd\mu(y)
 &= \int W(x,y) I_A(x) g_B(y)\dd\mu(x) \dd\mu(y)\ =\\
\int V_B(x) I_A(x)\dd\mu(x) 
&= \int V_B(x) g_A(x) \dd\mu(x)\ = \\
 \int W(x,y) g_A(x) g_B(y) \dd\mu(x) \dd\mu(y)
&= \int  \Wt(x,y) g_A(x) g_B(y) \dd\mu(x) \dd\mu(y).
 \end{align*}
 Observe that $g_A(x) g_B(y)$ is a conditional expectation
of $I_A(x)I_B(y)$ with respect to $\C A\otimes \C A$ while
$\Wt$ is measurable in this $\sigma$-algebra. Thus we can replace
$g_A(x) g_B(y)$ by $I_A(x)I_B(y)$ in the last integral, obtaining
$\int_{A\times B} \Wt$ as desired.
\end{proof}

Thus all assumptions of Lemma~\ref{lm:push} are satisfied and 
we obtain that $W\sim U^f$ for some almost kernel $U$ on 
$([0,1],\C B,\lambda,<)$. By Lemma~\ref{clean01}, we can change $U$ on a null
set so that $([0,1],\C B,\lambda,<,U)$ is a kernel. Clearly,
the equivalence $W\sim U^f$ is not affected by this. This finishes
the proof of Theorem~\ref{th:main}.

\section{A finite Szemer\'edi-type Regularity Lemma for
posets}\label{sec:finiteRL}

In this section we prove a Szemer\'edi-type
Regularity Lemma for posets, Theorem~\ref{thm:finRL}. 
(See Proemel, Steger, and
Taraz~\cite{promel+steger+taraz:01} and Patel~\cite{patel:thesis} for other
versions.)
We then show in
Section~\ref{ssec:altproof} that this result can be used to
answer
Janson's question. 

\def\cA{\mathcal{A}}\def\cL{\mathcal{L}}\def\cO{\mathcal{O}}\def\cP{\mathcal{P}}\def\cQ{\mathcal{Q}}\def\cR{\mathcal{R}}

Suppose that $(P,\prec)$ is a poset. For two disjoint sets $X,Y\subseteq P$ we
write $X\nprec Y$ if there are no $x\in X$ and $y\in Y$ such that $x\prec y$.
An (ordered) partition $\cP=(V_1,\dots,V_k)$ of the ground set $P$ is a
\emph{poset partition} if
 \begin{equation}\label{pp}
 V_i\nprec V_j,\quad\mbox{for all $1\le j<i\le k$.}
 \end{equation}
 In other words, every $\prec$-relation that involves vertices from two	
different parts goes ``forward''. 
We refer to members of $\cP$ as \emph{clusters}. 
Let us say that $\cR$ is a \emph{poset refinement} of $\cP$ if
$\cR$ is a poset partition that refines $\cP$ (that is,
for each $X\in\cR$ there exists $Y\in\cP$ such
that $X\subseteq Y$). The \emph{restriction} of $\cP$ to $X\subseteq P$
is $\cP|_X=(V_1\cap X,\dots,V_k\cap X)$. (For notational convenience,
we allow empty parts.)

Let $G=G_{P,\prec}$ be an (undirected) graph on the vertex set $P$ with
edge set 
 \begin{equation}\label{GPprec}
 E(G):=\{\,\{x,y\}\::\: x\prec y\mbox{ or } y\prec x\,\} .
 \end{equation}
 Clearly, if we know $G$ and a poset partition $\C P$, then we can reconstruct
$\prec$ except for pairs lying inside a part. The main idea behind
our Regularity Lemma is to find a poset partition of $P$ that is regular with
respect to $G$.

The following definitions apply to $A,B\subseteq P$.
The \emph{density} of
the pair $(A,B)$ is
$$ 
 d(A,B):=\frac{e(A,B)}{|A|\,|B|}:=\frac{|\{(x,y)\in A\times B: x\prec
y\}|}{|A|\,|B|},\qquad \mbox{if $A,B\not=\emptyset$,}
$$
 and $d(A,B):=0$ otherwise. The pair $(A,B)$
is called
\emph{$\e$-regular} if
$|d(A,B)-d(X,Y)|<\e$
for each $X\subseteq A$ and $Y\subseteq B$ with $|X|\ge \e |A|$ and
$|Y|\ge
\e |B|$. When we will apply the definition of $\e$-regularity to $(A,B)$, it
will always be the case that $B\nprec A$ (and we obtain the standard graph
definition). Also, let
 $$
  q(A,B):=\frac{|A|\,|B|}{|P|^2}\,
d^2(A,B).
 $$

For disjoint sets $V_1,\dots,V_k,U_1,\dots,U_m\subseteq P$, we define 
 \begin{eqnarray*}
 q(\,(V_1,\dots,V_k),(U_1,\dots,U_m)\,)&:=& \sum_{i=1}^k\sum_{j=1}^m
q(V_i,U_j),\\
 q(\,(V_1,\dots,V_k)\,)&:=&\sum_{i<j}q(V_i, V_j).
 \end{eqnarray*}
 The function $q$ is called the \emph{index}
and is crucial in the proof of the original Regularity Lemma.
Also, let $\I I_\e(\,(V_1,\dots,V_k)\,)$ be the set of pairs $(i,j)$ such that
$i<j$ and $(V_i,V_j)$ is not $\e$-regular.

The sizes of the clusters in
our Regularity Lemma can vary vastly (at least in our
proof).
This is why our next definition is slightly different from
the standard one. A poset partition $\cP=(V_1,\dots,V_k)$ of 
$P$ is \emph{$\e$-regular} if each $|V_i|\le\max(\e |P|,1)$ and
 $$
 \sum_{(i,j)\in\I I_\e(\cP)} |V_i|\,|V_j|\le \e {|P|\choose 2}.
 $$

\begin{theorem}[Regularity Lemma for Posets]\label{thm:finRL}
For each $\e>0$ there exists a number $M$ such that the following holds.
For each poset $(P,\prec)$ with a poset partition $\cP$ such that $|\cP|\le
1/\e$,
there exists a poset refinement $\cR$ of $\cP$ which is $\e$-regular 
and has at most $M$ parts.
\end{theorem}

\begin{remark}\label{rem:garbage}
It is important for our later application in Section~\ref{ssec:altproof} that
there is no garbage cluster in our partition. 
\end{remark}

We prove Theorem~\ref{thm:finRL} by following
Szemer\'edi's original proof of the Regularity Lemma for
graphs~\cite{szemeredi:76} (a
more accessible reference is for example~\cite[Section~7.4]{diestel:gt}). The
basic
idea is that if a current partition $\cP$ is not $\e$-regular then we can 
refine it so that $q(\cP)$ increases by at least $\delta$, where $\delta>0$
depends on $\e$ only. Since $q$ is always between $0$ and $1/2$, we reach an
$\e$-regular partition in at most $1/(2\delta)$ refinements. The following
index increment lemma estimates by
how much we can increase $q$ by subdividing one irregular pair
$(A,B)$.

\begin{lemma}\label{lem:posetirregularity}
Suppose that $(P,\prec)$ is a poset and $A,B\subseteq P$ are disjoint nonempty
sets. If $B\nprec A$ and $(A,B)$ is not $\e$-regular, then there are 
partitions $A=Z_1\cup Z_2$ and $B=Z_3\cup Z_4$ such that $Z_2\nprec
Z_1$, $Z_4\nprec Z_3$, and 
 \begin{equation}\label{eq:posetirregularity}
 q(\,(Z_1,Z_2),(Z_3,Z_4)\,)\ge q(A,B)+
\e^4\frac{|A|\,|B|}{n^2}.
 \end{equation}
\end{lemma}

\begin{proof}
Let $d:=d(A,B)$. Consider a witness of irregularity $(X,Y)$ of the pair
$(A,B)$. Assume without loss of generality that $d(X,Y)\ge d+\e$.

Iteratively, repeat the following as long as possible: replace
some $x\in X$ by some $y\in A\setminus X$ with $y\prec x$. Clearly, this
operation preserves the size of $X$ and cannot decrease $d(X,Y)$. Also, we
have
to stop at some point. Let $Z_1$ be the final $X$ and let $Z_2:=A\setminus Z_1$.

Similarly, replace $Y\subseteq B$ by an up-closed subset
$Z_4\subseteq B$ such that $|Z_4|=|Y|$ and $d(Z_1,Z_4)\ge d(Z_1,Y)\ge d+\e$.
Let $Z_3:=B\setminus Z_4$.  Of course, we have that $Z_2\nprec Z_1$ and $Z_4\nprec Z_3$.

Note that $(Z_1,Z_4)$ demonstrates that the pair $(A,B)$ is not $\e$-regular. 
Since $B\nprec A$, such density statements also hold with respect to the (undirected)
graph $G_{P,\prec}$ that was defined in (\ref{GPprec}). Thus the standard
index estimates from graph theory apply here. In particular, 
the proof of Lemma~7.4.3 in~\cite{diestel:gt}
shows that (\ref{eq:posetirregularity}) holds whenever $(Z_1,Z_4)$ is a witness of $\e$-irregularity.\end{proof}

\begin{proof}[Proof of Theorem~\ref{thm:finRL}.] 
Let $s:=\lceil 2/\e^5\rceil$,  $k_0:=\lceil 2/\e\rceil$, and inductively
for
$t=0,\dots,s-1$, let $k_{t+1}:=k_{t}\, 2^{k_{t}-1}$. We claim that $M=k_s$
suffices.

Suppose that $n:=|\cP|> 1/\e$ for otherwise we can let $\cR$ be a partition into
singletons. 

Initially, let $\cR_0$ be an arbitrary poset refinement of $\cP$ such
that $|\cR_0|\le k_0$
and each part
has at most $\e n$ vertices. 

Iteratively, for
$t=0,1,\dots$, we repeat the following procedure. Let $\cR_t=(V_1,\dots,V_k)$.
If
$\cR_t$ is $\e$-regular then we stop and output $\cR_t$; so suppose otherwise.
Let $\cR':=\cR_t$. We modify $\cR'$ by using another (embedded) iterative
procedure.
Namely,
in turn for
each $(i,j)\in\I I_\e(\cR_t)$,
we take the partitions $V_i=Z_{1ij}\cup Z_{2ij}$ and $V_j=Z_{3ij}\cup Z_{4ij}$
returned by Lemma~\ref{lem:posetirregularity} and replace every $X\in\cR'$
by $X\cap Z_{1ij},\dots, X\cap Z_{4ij}$, with these four parts coming
in the specified order. Clearly, $\cR'$ is still a poset partition. Once we
have processed all elements of $\I I_\e(\cR_t)$, we let $\cR_{t+1}:=\cR'$.

In order to estimate how $q$ changes, let us write
 \begin{equation}\label{diffq}
 q(\cR_{t+1})-q(\cR_t)\ge \sum_{1\le i<j\le k}
\Big(q(\cR_{t+1}|_{V_i},\cR_{t+1}|_{V_j})-q(V_i,V_j)\Big),
 \end{equation}
 where the inequality comes from discarding the sum $\sum_{i=1}^k q(\cR_{t+1}|_{V_i})\ge 0$.
 We can estimate each summand corresponding to $(i,j)\in\I I_\e(\cR_t)$ by
passing from $q(V_i,V_j)$ first to $q(\,(Z_{1ij},Z_{2ij}),(Z_{3ij},Z_{4ij})\,)$
and then to $q(\cR_{t+1}|_{V_i},\cR_{t+1}|_{V_j})$. The first step
increases $q$ as specified by
Lemma~\ref{lem:posetirregularity}. The second
step has non-negative effect by~\cite[Lemma~7.4.2]{diestel:gt}. Each other
term in the right-hand side of (\ref{diffq}) is non-negative, again 
by~\cite[Lemma~7.4.2]{diestel:gt}. Since $\cR_t$ is not $\e$-regular, we
conclude that
 \begin{equation}\label{diffq2}
 q(\cR_{t+1})-q(\cR_t)\ge \frac{\e^4}{n^2} 
\sum_{(i,j)\in\I I_{\e}(\cR_t)} |V_i|\,|V_j| > \frac{\e^4}{n^2}\;
\e {n\choose 2}\ge \frac{\e^5}{4}.
 \end{equation} 

Trivially, $0\le q(\cP)\le 1/2$ for any partition $\cP$.
By~(\ref{diffq2}), we repeat the iteration procedure at most $s$ times before we
reach an $\e$-regular poset partition. 
 As each part of $\cR^t$
is split into at most $2^{|\cR_t|-1}$ parts, we have that 
$|\cR^{t+1}|\le|\cR_t|\,2^{|\cR_t|-1}$. Thus the final
partition has at most $M$
parts, as required.\end{proof}

If we do not know $\prec$ but know an $\e$-regular partition
$\cR=(V_1,\dots,V_k)$ and the densities between all pairs of parts, then we can
still derive various information about the poset $P$. For example, given two
subsets $S,T\subseteq P$, one would expect to see approximately
 $$
 e'(S,T):=\sum_{i<j} d(V_i,V_j)\,|V_i\cap S|\,|V_j\cap T|
 $$
 directed arcs from $S$ to $T$. Indeed, this is the case for posets.

\begin{lemma} Given the above assumptions, we have
 \beq\label{ee'}
 |e(S,T)-e'(S,T)|\le 3\e {|P|\choose 2}.
 \eeq
 \end{lemma}

\begin{proof} 
 Let $n:=|P|$. Assuming the worst-case scenario, the edges inside a part or
inside a non-$\e$-regular pair contribute at most $\e{n\choose
2}+\e{n\choose 2}$ to the left-hand side of~(\ref{ee'}).
 For every $\e$-regular pair $(V_i,V_j)$ with $i<j$, we have 
 $$
 \Big|\,e(V_i\cap S, V_j\cap T)-d(V_i,V_j)\,|V_i\cap S|\,|V_j\cap T| \,\Big|
\le \e |V_i|\,|V_j|.
 $$
 Indeed, if $|V_i\cap S|\,|V_j\cap T|\le \e |V_i|\,|V_j|$, then we are trivially
done;
 otherwise both $S$ and $T$ take
more than $\e$-proportion of respectively $V_i$ and $V_j$ and the bound
follows by the $\e$-regularity of $(V_i,V_j)$. 
Thus the
aggregate contribution of $\e$-regular pairs to (\ref{ee'}) is at most~$\e
{n\choose 2}$.
\end{proof}

\section{An alternative proof of Theorem~\ref{our}}\label{ssec:altproof}

Let $\{(P_n,\prec_n)\}_{n\in\I N}$ be a convergent sequence
of posets. We have to construct a
kernel $([0,1],\C B,\lambda,<,U)$ such that for every poset $P$ we have
 \beq\label{ConvP}
 t(P,U)=\lim_{n\to\infty} t(P,P_n).
 \eeq
 We construct $U$ following closely the analogous construction of Lov\'asz and
Szegedy \cite[Theorem~2.4]{lovasz+szegedy:06} (see also~\cite[Theorem
5.1]{lovasz+szegedy:07:gafa}). In brief, the proof proceeds
by finding
a $\frac1k$-regular partition $\cP_{n,k}$ of $P_n$ with the number of parts
bounded by a function of $k$ only. Then we construct a step-function
$W_{n,k}:[0,1]^2\to[0,1]$ that encodes the part ratios and
densities of $\cP_{k,n}$. Since the ``complexity'' of $W_{n,k}$ is bounded
by a function of $k$, a diagonalisation process gives a subsequence
$\{P_{n_i}\}_{i\in\I N}$ such that, for every $k$, we have 
$W_{n,k}\to U_k$ a.e.\ for some $U_k:[0,1]^2\to[0,1]$.
Additionally, when we choose our partitions $\cP_{k,n}$,
we can assume that they are nested for each $n$. This allows us to write
$U_{k-1}$ as a conditional expectation of $U_k$ and conclude that
$\{U_k\}_{k\in\I N}$
converges to some $U$ a.e. Finally, we need to apply Lemma~\ref{clean01} to
tranform an almost kernel $U$ into a kernel.

Let us give more details. Let $m_1=1$ and inductively for $k=2,3,\dots$ let
$m_{k}$ be sufficiently large such that every poset partition with at most
$m_{k-1}$
parts admits a $\frac1k$-regular poset refinement with $m_{k}$ parts. Such a
number exists by Theorem~\ref{thm:finRL}. (Recall that we allow
empty parts.)
For each $n\in\I N$, let $\cP_{n,1}:=(P_n)$ be the trivial partition and then
inductively for $k=2,3,\dots$ let
 \beq\label{cPnk}
 \cP_{n,k}=(V_{n,k,1},\dots,V_{n,k,m_k})
 \eeq 
 be a $\frac1k$-regular poset partition of
$(P_n,\prec_n)$ that refines $\cP_{n,k-1}$. This nestedness allows us
for each $n$, to choose a total ordering $\prec'_n$ of $(P_n,\prec_n)$
which is \emph{compatible} with every poset partition $\cP_{n,k}$ (that is,
$V_{n,k,i}\nprec_n' V_{n,k,j}$ whenever $i>j$). By relabelling, let us assume
that $P_n=\{1,\dots,|P_n|\}$ and $\prec_n'$ is the standard order.

Already at this point, it makes sense to start operating with
functions. Let
$W_n:[0,1]^2\to\{0,1\}$
be the step-function that encodes the $\prec_n$-relation in the obvious way:
$W_n$ is constant on $[\frac {i-1}v,\frac{i}v)\times [\frac
{j-1}v,\frac{j}v)$,
where $v:=|P_n|$, and assumes
value 1 there if and only if $i\prec_n j$. It is easy to see
that
 $$
 t(P,P_n)=t(P,W_n),\quad \mbox{for every poset
$P$}
 $$
 where we view $W_n$ as a
kernel on $([0,1],\C B,\lambda,<)$.

Let $\cP_{n,k}'=(V_{n,k,1}',\dots,V_{n,k,m_k}')$ be
the partition of $[0,1]$ into consecutive intervals corresponding to
(\ref{cPnk}). (Thus, for example,
$\lambda(V_{n,k,i}')=|V_{n,k,i}|/|P_n|$.) Let $W_{n,k}$ be the step-function
on $\cP_{n,k}'\times \cP_{n,k}'$, whose steps correspond to the parts of 
$\cP_{n,k}$ and whose values
correspond to densities between parts. We can write this more compactly as
 $$
 W_{n,k}\sim \E(W_n|\sigma(\cP_{n,k}')),
 $$
 a conditional expectation of $W_n$ with
respect to the (finite) $\sigma$-algebra generated by $\cP_{n,k}'$.  Since
$\sigma(\cP_{n,1}')\subseteq\sigma(\cP_{n,2}')\subseteq\dots$, we have
 $$
 W_{n,k} \sim\E(W_{n,k+1}|\,\sigma(\cP_{n,k}')),\quad k\ge 1,
 $$
 which translates into the combinatorially obvious fact that the densities of
$\cP_{n,k}$ can be obtained by averaging over the densities in the
finer partition $\cP_{n,k+1}$.

Since each $W_{n,k}$ can be described by specifying part sizes and densities
(which involves at most $m_k+{m_k\choose 2}$ reals in $[0,1]$),
the standard diagonalisation process gives a subsequence $\{n_i\}_{i\in
\I N}$ 
such that these parameters converge for every $k$. Thus
$W_{n_i,k}\to U_k$ a.e.\ for some step-function $U_k$ with $m_k$ steps
that are intervals and are ordered as
$\cP_k'=(V_{k,1}',\dots,V_{k,m_k}')$. Since
$\{P_n\}_{n\in\I N}$ is convergent, passing to a subsequence does not
affect (\ref{ConvP}); thus we can assume that $\{W_{n,k}\}_{n\in\I N}$ itself
a.e.\ converges to $U_k$. Clearly,
$\sigma(\cP_1')\subseteq\sigma(\cP_2')\subseteq\dots$ a.e.\ and 
 $$
 U_k\sim\E(U_{k+1}|\,\sigma(\cP_k')).
 $$
 Thus, by the Martingale
Convergence Theorem (see e.g.\ \citeB{Theorem~10.3.3}), $U_k\to U$ a.e.\ for
some $U:[0,1]^2\to [0,1]$. 

The obtained function $U$, as the a.e.\ pointwise
limit of Borel functions, is Borel a.e. Clearly, the kernel axioms hold for
$([0,1],\C B,\lambda,U)$ for
all inputs that do not require the evaluation of $U$ on a point of
 $$
 X:=\big\{(x,y)\in [0,1]^2: U_k(x,y)\not\to U(x,y)\mbox{ or $\exists\, k$ }
W_{n,k}(x,y)\not\to U_k(x,y)\big\},
 $$
 the set where some convergence fails. Since
$X$ has measure zero, $U$ is an almost kernel. By applying
Lemma~\ref{clean01}, we can assume that $U$ is a kernel.

It remains to show that (\ref{ConvP}) holds. The \emph{cut-norm} of a
bounded measurable function $W:[0,1]^2\to\I R$ is defined by
 \beq\label{cutnorm}
 \|W\|_{\Box}=\sup_{S,T\in\C B} \left|\,\int_{S\times T}
W(x,y)
\dd\lambda(x)\dd\lambda(y)\,\right|.
 \eeq
 
\begin{claim}\label{cl:cut}$\|W_n-W_{n,k}\|_\Box\le \frac{5}{2k}$
for any $k,n\in\I N$.\end{claim}

\begin{proof}[Proof of Claim.] Let $W:=W_n-W_{n,k}$. Assume that $v:=|P_n|> k$
for otherwise there is nothing do to as $W=0$.

Observe that, up to an additive error $\frac1{v}$, it is enough to
consider those $S$ and $T$ in (\ref{cutnorm}) that are unions of intervals
$V_i:=[\frac{i-1}v,\frac{i}v)$ for $i\in [v]$. Indeed, fix any
$S,T\in\C B$ with, say, $\int_{S\times T}W\ge 0$ and take $i\in[v]$ one by one.
If we modify $S$ and $T$ inside $V_i$, then
the integral of $W$ over
 $$
 \big((V_i\times V_i^c)\cup (V_i^c\times V_i)\big) \cap (S\times T)
 $$
 is a linear function of $\lambda(V_i\cap S)$ and $\lambda(V_i\cap T)$.
Thus we can make each of these to belong to $\{0,1/v\}$ without decreasing the
above contribution. Updating $S$ and $T$
accordingly, we decrease $\int_{S\times T}W$ by at most $\int_{V_i\times
V_i}|W|\le 1/v^2$. When we have iteratively processed
all $i\in[v]$,
both $S$ and $T$ have the desired form.

Thus, by (\ref{ee'}), we obtain the required:
 $$
 \|W\|_{\Box}\le \frac3k{v\choose 2}\,\frac1{v^2} + \frac1{v} \le
\frac{3}{2k}+\frac1{k}=\frac{5}{2k}.
 $$
 \end{proof}

Now, we are ready to verify (\ref{ConvP}). Take any poset $(P,\prec)$ and
$\e>0$.
Let $m:=e(G_{P,\prec})$ be the number of pairs in $\prec$.

Since we deal with
bounded measurable functions, all convergences also hold in the
$\ell_1$-space on $([0,1]^2,\C B,\lambda)$
by~\citeB{Theorem~2.2.3}. Thus there is $k\ge \frac{15m}{2\e}$
such that $\|U-U_{k}\|_1\le \frac{\e}{3m}$ and, fixing this $k$, there is $n_0$
such that
$\|U_{k}-W_{n,k}\|_1\le \frac{\e}{3m}$ for all $n\ge n_0$.
Clearly, $\|f\|_\Box\le \|f\|_1$ for any integrable $f$. Thus, by the Triangle
Inequality and
Claim~\ref{cl:cut}, we have that for all $n\ge n_0$
 \begin{eqnarray*}
 \|U-W_{n}\|_\Box &\le& \|U-U_k\|_\Box + \|U_k-W_{n,k}\|_\Box +
\|W_{n,k}-W_n\|_\Box\\ 
 &\le & \|U-U_k\|_1 + \|U_k-W_{n,k}\|_1 + \frac{5}{2k}\ \le\ 
\frac{\e}{m}.
 \end{eqnarray*}
  By  \citeJ{Lemma~6.4}, we have that $|t(P,U)-t(P,W_n)|\le m\,\|U-W_n\|_\Box\le
\e$. Since $\e$ and $P$ were
arbitrary, (\ref{ConvP}) follows. 

Summarising, $([0,1],\C B,\lambda,<,U)$ is a 
kernel that establishes Theorem~\ref{our}.

\begin{remark} An alternative way to proving that the densities of $F$
in $W_n$ and
$W_{n,k}$ are close  is to
adopt
the Counting Lemma (see e.g.\ \cite[Theorem~5]{simonovits+sos:91}) to
our settings. We do not see any principal difficulties
here but we expect that the error term would be
larger.\end{remark}

\begin{remark} In the above proof it is not generally true that
$\|W_{n}-W_{n,k}\|_1$ is small for sufficiently
large $k$: for example, $W_{n,k}$ may be strictly between $0$ and $1$ on
a set of positive measure
(while $W_n$ is always $\{0,1\}$-valued). \end{remark}

\section{Concluding remarks}\label{conclusion}

There are two natural ways to extend the definition of
convergence to the case when the poset orders do not tend
to infinity. One is to just use~(\ref{eq:conv}). Another, adopted by
Janson~\citeJ{Definition~3.2},
is to say that $\{P_n\}_{n\in \I N}$ with $|P_n|\not\to\infty$ is convergent if
the
sequence is eventually constant (up to isomorphism). The choice of which one
to use (or none) is more a matter of convenience. For example, this choice 
may depend on whether we want the ``limits'' of $(P,P,\dots)$ and
$(P^{(1)},P^{(2)},\dots)$ to be the same or not. Here the \emph{blow-up}
$P^{(k)}$
of $P$ is obtained by cloning  $k$ times each vertex of $P$; obviously,
$t(Q,P)=t(Q,P^{(k)})$ for every poset $Q$.
Since all results stated in the introduction can be
trivially reduced to the case $|P_n|\to\infty$ by blowing posets up, we decided
to
use Definition~\ref{conv}.

Of course, the
assumption that $(S,\C F,\mu)$ is atomless is necessary in
Theorems~\ref{th:main} and \ref{Total}. This assumption can be removed if we are
allowed to modify $([0,1],\C B,\lambda,{<})$ by shifting positive measure to a
some countable subset $X\subseteq [0,1]$, where $X$ depends on $(S,\C
F,\mu,\tl)$.
However, we believe that the versions presented in the introduction are neater.

We cannot require in Theorems~\ref{th:main} and~\ref{Total} that $f$
preserves \emph{every} relation (i.e.\ that
the set in (\ref{wrong}) is empty) as the following example demonstrates. Let
$S:=[0,1)$ with the Lebesgue measure $\lambda$ on the Borel $\sigma$-algebra
$\C B$. Fix an irrational number $\tau$. Let  $T:S\to S$ map $x$ to $x+\tau\pmod
1$. If we view $S$ as a circle, then $T$ is an aperiodic rotation.
Define $x\tl y$ if there is $k\in \I
N$ with $y=T^k(x)$. The constructed relation $\tl$ is a Borel subset of
$S^2$ (of measure zero).
Let us suppose on the contrary that there is an inclusion $f:(S,\C
B,\lambda,\tl)\to ([0,1],\C B,\lambda,{<})$ such that the set in
(\ref{wrong}) is empty. Let $A:=f^{-1}(\,[0,\frac12]\,)$. Since $A$ is a
down-closed set with respect to $<$, we have that $T^{-1}(A)\subseteq A$.
Since $T$ is measure-preserving, we conclude that $T^{-1}(A)\sim A$. However,
this contradicts the well-known fact (see e.g.~\citeB{Example 10.9.9}) that $T$
is ergodic. Alternatively,
let $B:=\cap_{k=1}^\infty T^{-k}(A)$. Then $B$ is a measurable set such
that $T^{-1}(B)=B$ (exactly) and $\mu(B)=1/2$ (by
$\sigma$-additivity). The same applies to $B^c$.
By taking density points $x$ and $y$ of $B$ and $B^c$ respectively and
a sequence
of $k$ such that $T^k(x)\to y$, one readily arrives at the desired
contradiction.
 
Also, the assumption that $W$ is strict in Theorem~\ref{th:main} is needed. 
For example, take $[0,1]^2$ with the Legesgue measure on the Borel sets
and let $(x,y)\tl(x',y')$ if $x<x'$. Let, for example, $W((x,y),(x',y'))$ be
$y'$
if $x'>x+1/2$ and 0 otherwise. It is easy to see that
every inclusion of this ordered probability space into the unit interval
is a.e.\ equal to the projection onto the first coordinate. 
However, $W((x,y),(x',y'))$ is essentially non-constant on $(x,x')$-slices for
$x'>1/2+x$ and 
thus cannot be equivalent to some pull-back.

\section*{Acknowledgements}

The authors thank the anonymous referee for helpful comments.

\bibliographystyle{amsplain}

\providecommand{\bysame}{\leavevmode\hbox to3em{\hrulefill}\thinspace}
\providecommand{\MR}{\relax\ifhmode\unskip\space\fi MR }
\providecommand{\MRhref}[2]{%
  \href{http://www.ams.org/mathscinet-getitem?mr=#1}{#2}
}
\providecommand{\href}[2]{#2}

\end{document}